\documentclass[a4paper,twoside,11pt,reqno]{amsart}
\usepackage{fullpage}
\usepackage[T1]{fontenc}
\usepackage[utf8]{inputenc}

\usepackage{hyperref}
\usepackage[slantedGreeks, partialup, noDcommand]{kpfonts}
\usepackage{graphicx}
\usepackage[mathscr]{euscript}

\hypersetup{ocgcolorlinks=true,allcolors=testc}
\hypersetup{
     colorlinks   = true,
     citecolor    = black
}
\hypersetup{linkcolor=black}
\hypersetup{urlcolor=black}

\newcommand{\eqdef}{\stackrel{\mathrm{def}}{=}}

\usepackage[dvipsnames]{xcolor}

\newtheorem{theorem}{Theorem}
\newtheorem{lemma}[theorem]{Lemma}

\newtheorem{proposition}[theorem]{Proposition}
\newtheorem{corollary}[theorem]{Corollary}

\newtheorem{remark}[theorem]{Remark}
\newtheorem{definition}[theorem]{Definition}
\newtheorem*{definition*}{Definition}

\newcommand*\diff{\mathop{}\!\mathrm{d}}
\newcommand{\R}{\mathbb{R}} 

\newcommand{\N}{\mathbb{N}}

\newcommand{\e}{\varepsilon}

\usepackage{dutchcal}

\newcommand{\MM}{\mathsf{M}}

\newcommand{\mb}{\mathbb}
\newcommand{\ms}{\mathscr}
\newcommand{\msf}{\mathsf}

\newcommand{\mf}{\mathfrak}

\newcommand{\T}{\mathsf{T}}
\newcommand{\s}{\mathsf{S}}

\title{Some geometric applications of the discrete heat flow}

\author{Alexandros Eskenazis}
\address{CNRS, Institut de Math\'ematiques de Jussieu, Sorbonne Universit\'e, France and Trinity College, University of Cambridge, UK.}
\email{alexandros.eskenazis@imj-prg.fr, ae466@cam.ac.uk}

\begin{document}

\maketitle
\vspace{-3mm}

\begin{abstract}
We present two geometric applications of heat flow methods on the discrete hypercube $\{-1,1\}^n$. First, we prove that if $X$ is a finite-dimensional normed space, then the bi-Lipschitz distortion required to embed $\{-1,1\}^n$ equipped with \mbox{the Hamming metric into $X$ satisfies}
$$\mathsf{c}_X\big(\{-1,1\}^n\big) \gtrsim \sup_{p\in[1,2]} \frac{n}{\mathsf{T}_p(X) \min\{n,\mathrm{dim}(X)\}^{1/p}},$$
where $\mathsf{T}_p(X)$ is the Rademacher type $p$ constant of $X$. This estimate yields a mutual refinement of distortion lower bounds which follow from works of Oleszkiewicz (1996) and Ivanisvili, van Handel and Volberg (2020) for low-dimensional spaces $X$. The proof relies on an extension of an important inequality of Pisier (1986) on the biased hypercube combined with an application of the Borsuk--Ulam theorem from algebraic topology. Secondly, we introduce a new metric invariant called metric stable type as a functional inequality on the discrete hypercube and prove that it coincides with the classical linear notion of stable type for normed spaces.  We also show that metric stable type yields bi-Lipschitz nonembeddability estimates for weighted hypercubes.
\end{abstract}

\bigskip

{\footnotesize

}



\section{Introduction}

Let $\{-1,1\}^n$ be the $n$-dimensional discrete hypercube equipped with the Hamming metric
\begin{equation}
\forall\ x,y\in\{-1,1\}^n,\qquad \rho(x,y) = \frac{1}{2} \sum_{i=1}^n |x(i)-y(i)|,
\end{equation}
where $x=(x(1),\ldots,x(n))$ and $y=(y(1),\ldots,y(n))$.  The purpose of the present paper is to investigate certain metric properties of the Hamming cube via heat flow methods.


\subsection{Dimensionality of Hamming metrics}

If $(\MM,d_\MM)$ is a metric space and $(Y,\|\cdot\|_Y)$ is a normed space,  we say that $\MM$ embeds into $Y$ with bi-Lipschitz distortion at most $D\in[1,\infty)$, if there exists a mapping $f:\MM\to Y$ satisfying
\begin{equation} \label{eq:bi-Lip}
\forall \ p,q\in\MM,\qquad d_\MM(p,q) \leq \|f(p)-f(q)\|_Y \leq D d_\MM(p,q).
\end{equation}
The least $D\geq1$ for which such an embedding exists will be denoted by $\msf{c}_Y(\MM)$. The rapidly growing field of metric dimension reduction aims to uncover conditions under which given families of metric spaces admit (or do not admit) embeddings into low-dimensional normed spaces with prescribed properties.  Without attempting to survey this vast area,  we note that important contributions have been made on low-dimensional embeddings of finite subsets of Hilbert space \cite{JL84},  arbitrary finite metric spaces \cite{JLS87,AdRRP92,Mat92,Mat96},  discrete hypercubes \cite{Ole96,LMN05},  diamond graphs \cite{BC05,LN04,NPS20},  Laakso graphs \cite{GKL03,LMN05}, ultrametric spaces \cite{BM04}, series-parallel graphs \cite{BKL07}, recursive cycle graphs \cite{ACNN11}, Heisenberg-type metrics \cite{LN14,NY22},  $\ell_p$ variants of thin Laakso structures \cite{BGN15,BSS21} and expander graphs \cite{Nao17,Nao21}. We refer to the survey \cite{Nao18} for more bibliographic information \mbox{and to \cite{Ind01,Lin02,Vem05,AIR18} for a sample of algorithmic applications.}

The first to study the bi-Lipschitz embeddability of hypercubes into normed spaces was Enflo.  In the seminal work \cite{Enf69},  he introduced the notion of {\it roundness} of a metric space and used it to show that any embedding of the Hamming cube $\{-1,1\}^n$ into an $L_p(\mu)$ space, where $p\in[1, 2]$,  incurs bi-Lipschitz distortion at least $n^{1-1/p}$ (see also \cite{Enf70, Enf78} for additional early results along these lines). More specifically, Enflo proved that if $p\in[1,2]$, then any mapping of the form $f:\{-1,1\}^n\to L_p(\mu)$ satisfies the estimate
\begin{equation} \label{eq:enflo}
\int_{\{-1,1\}^n} \big\| f(x)-f(-x)\big\|_{L_p(\mu)}^p \,\diff\sigma_n(x) \leq \sum_{i=1}^n \int_{\{-1,1\}^n} \big\|f(x)-f\big(x(1),\ldots,-x(i),\ldots,x(n)\big)\big\|_{L_p(\mu)}^p\,\diff\sigma_n(x),
\end{equation}
where $\sigma_n$ is the uniform probability on $\{-1,1\}^n$.  This readily implies that if $f$ has bi-Lipschitz distortion $D$, then $D\geq n^{1-1/p}$. In the follow-up work \cite{Enf78}, he raised an influential problem by asking for which normed spaces $(X,\|\cdot\|_X)$, inequality \eqref{eq:enflo} is satisfied for $X$-valued functions $f$ up to a multiplicative constant $T$, independent of the choice of $f$ or the dimension $n$. Restricting this requirement to linear functions $f(x) = \sum_{i=1}^n x_i v_i$, one recovers the necessary condition
\begin{equation} \label{eq:type}
\int_{\{-1,1\}^n} \Big\| \sum_{i=1}^n x_i v_i\Big\|_X^p\,\diff\sigma_n(x) \leq T^p \sum_{i=1}^n \|v_i\|_X^p,
\end{equation}
which must be satisfied for every $n\in\N$ and vectors $v_1,\ldots,v_n\in X$. If a normed space $X$ satisfies \eqref{eq:type}, we say that $X$ has \emph{Rademacher type} $p$ and the least constant $T$ is denoted by $\T_p(X)$. After decades of substantial efforts (see \cite{BMW86, Pis86,NS02,HN13,Esk21}),   Ivanisvili, van Handel and Volberg resolved Enflo's problem in the breakthrough work \cite{IVV20} by proving the sufficiency of this condition, namely that any normed space of Rademacher type $p$ also has Enflo's nonlinear type $p$. Consequently, any bi-Lipschitz embedding of $\{-1,1\}^n$ into a normed space $X$ of Rademacher type $p$ incurs distortion at least a constant multiple of $\T_p(X)^{-1}n^{1-1/p}$. We note in passing that, conversely, a classical theorem of Pisier \cite{Pis73} implies that if $X$ does not have type $p$ for any $p>1$, then $\{-1,1\}^n$ embeds into $X$ with bi-Lipschitz distortion at most $1+\e$, for any $\e>0$.

Independently of this line of research, the beautiful (but perhaps overlooked) work \cite{Ole96} of Oleszkiewicz established a nonembeddability result for discrete hypercubes in the context of dimensionality reduction. Following Ball, Carlen and Lieb \cite{BCL94}, we say that a normed space is $p$-{\it uniformly smooth}, where $p\in[1,2]$, if there exists a constant $S>0$ such that
\begin{equation} \label{eq:smooth}
\forall \ x,y\in X,\qquad \frac{\|x\|_X^p+\|y\|_X^p}{2}\leq \Big\|\frac{x+y}{2}\Big\|_X^p +S^p \Big\|\frac{x-y}{2}\Big\|_X^p;
\end{equation}
the least such constant $S$ is denoted by $\s_p(X)$.  A well-known tensorization argument due to Pisier \cite{Pis75} shows that $\T_p(X) \leq \s_p(X)$, yet there exist examples of normed spaces $X$ for which $\T_p(X)<\infty$ whereas $\s_p(X)=\infty$ for $p\in(1,2]$, see \cite{Pis75b,Jam78,PX87}. The main result of Oleszkiewicz's paper \cite{Ole96} asserts that the distortion required to embed $\{-1,1\}^n$ into a finite-dimensional normed space $X$ satisfies
\begin{equation} \label{eq:ole}
\msf{c}_X\big(\{-1,1\}^n\big) \geq \sup_{p\in[1,2]} \frac{n}{\s_p(X) \min\{n,\dim(X)\}^{1/p}}
\end{equation}
(see also \cite{BG81} for a precursor of this result for linear embeddings). This substantially improves the bound $\msf{c}_X(\{-1,1\}^n)\gtrsim \sup_{p\in[1,2]} \T_p(X)^{-1}n^{1-1/p}$ which follows from \cite{IVV20}, at least for spaces $X$ with $p$-smoothness constant $\s_p(X)\asymp 1$ and dimension $\dim(X)<\!\!\!< n$.

The first goal of the present paper is to revisit the technique used for Oleszkiewicz's nonembeddability theorem \cite{Ole96}, in particular proving the following mutual refinement of his result and of the recent work of Ivanisvili, van Handel and Volberg \cite{IVV20}.

\begin{theorem} \label{thm:main}
Let $(X,\|\cdot\|_X)$ be a finite-dimensional normed space.  Then, for any $n\geq1$, we have
\begin{equation} \label{eq:main}
\msf{c}_X\big(\{-1,1\}^n\big) \gtrsim \sup_{p\in[1,2]} \frac{n}{\T_p(X) \min\{n,\dim(X)\}^{1/p}}
\end{equation}
\end{theorem}

We emphasize that, in contrast to Oleszkiewicz's bound \eqref{eq:ole}, Theorem \ref{thm:main} captures more accurately the nonembeddability of the hypercube into (finite-dimensional subspaces of) normed spaces which have Rademacher type $p$ but are not $r$-smooth for any $r\in(1,2]$, see \cite{Jam78,PX87}.


\subsubsection*{About the proof} 

Theorem \ref{thm:main} is proven by a combination of semigroup tools with a clever topological trick of \cite{Ole96}. More specifically, let $f:\{-1,1\}^n\to X$ be a function, where $X$ is a $d$-dimensional normed space and $d<n$. An application of the Borsuk--Ulam theorem \cite{Mat03} for the unique multilinear extension of $f$ implies that there exists a subset $\sigma\subseteq\{1,\ldots,n\}$ with $|\sigma|=d$, a \emph{product} measure $\nu$ on $\{-1,1\}^{\sigma}$ and a point $w\in\{-1,1\}^{\sigma^c}$ such that
\begin{equation} \label{eq:top-info}
\int_{\{-1,1\}^\sigma} f(x,w) \,\diff\nu(x) = \int_{\{-1,1\}^\sigma} f(-x,-w) \,\diff\nu(x).
\end{equation}
Then, a Poincar\'e inequality \`a la Enflo for the product measure $\nu$ (instead of the uniform measure $\sigma_d$) on the $d$-dimensional subcubes $\{-1,1\}^\sigma\times\{w\}$ and $\{-1,1\}^\sigma\times\{-w\}$ yields the distortion bounds of Theorem \ref{thm:main} (see Theorem \ref{thm:poincare} and also equation \eqref{eq:real-conclusion} below).

In the case of $p$-uniformly smooth spaces, Oleszkiewicz \cite{Ole96} used \eqref{eq:top-info} and a bootstrap argument for the Lipschitz constant of $f$, based on the two-point inequality \eqref{eq:smooth}, to obtain \eqref{eq:ole}. In our case, the biased Poincar\'e inequality which will yield \eqref{eq:main} is an extension of an inequality for the uniform measure that was proven in \cite{IVV20}. The key technical contribution of \cite{IVV20} was a novel representation of the time derivative of the heat flow on $\{-1,1\}^n$. Instead, we consider a Markov process having the product measure $\nu$ as stationary measure (see Section \ref{sec:prel}) and prove a formula for the time derivative of the corresponding semigroup (see Proposition \ref{prop:ident}) which extends the formula of \cite{IVV20} (see also \eqref{eq:ivv} below). Due to the fact that our product measure $\nu$ is no longer the stationary measure of the random walk on a group, the resulting identity lacks some homogeneity properties that were used in \cite{IVV20}, but nevertheless it is sufficient for the proof of the biased Poincar\'e inequality which is needed for our geometric application.

\subsection{A Pisier--Talagrand inequality on the biased cube} 

For $\alpha\in(0,1)$ we denote by $\mu_\alpha$ the $\alpha$-biased probability measure on $\{-1,1\}$ with $\mu_\alpha\{1\}=\alpha$ and $\mu_\alpha\{-1\}=1-\alpha$. The proof of Theorem \ref{thm:main} yields as a consequence the following extension and refinement of Pisier's inequality \cite{Pis86}. 

\begin{theorem} \label{thm:pisier}
For every $p\in[1,\infty)$ and $\alpha\in(0,1)$, there exist $\msf{K}_{p,\alpha}, \msf{C}_\alpha\in(0,\infty)$ such that the following holds.  For any normed space $(X,\|\cdot\|_X)$ and any $n\in\N$, every function $f:\{-1,1\}^n\to X$ satisfies
\begin{equation} \label{eq:p1}
\begin{split}
\bigg\| f - \int_{\{-1,1\}^n} f\,\diff\mu_\alpha^n\bigg\|_{L_p(\log L)^{p/2}(\mu_\alpha^n;X)} \leq \msf{K}_{p,\alpha} & \left( \int_{\{-1,1\}^n} \Big\| \sum_{i=1}^n \delta_i \partial_i^\alpha f\Big\|_{L_p(\mu_\alpha^n;X)}^p \,\diff\sigma_n(\delta)\right)^{1/p}
\\ & + \msf{C}_{\alpha}(\log n+1) \int_{\{-1,1\}^n} \Big\|\sum_{i=1}^n\delta_i \partial_i^\alpha f\Big\|_{L_1(\mu_\alpha^n;X)}\,\diff\sigma_n(\delta).
\end{split}
\end{equation}
If additionally $X$ is assumed to be of finite cotype, then there exists $\msf{K}_{p,\alpha}(X)\in(0,\infty)$ such that
\begin{equation} \label{eq:p2}
\bigg\| f - \int_{\{-1,1\}^n} f\,\diff\mu_\alpha^n\bigg\|_{L_p(\log L)^{p/2}(\mu_\alpha^n;X)} \leq \msf{K}_{p,\alpha}(X)  \left( \int_{\{-1,1\}^n} \Big\| \sum_{i=1}^n \delta_i \partial_i^\alpha f\Big\|_{L_p(\mu_\alpha^n;X)}^p \,\diff\sigma_n(\delta)\right)^{1/p}.
\end{equation}
\end{theorem}

The (standard) definitions of Orlicz norms, cotype and discrete derivatives will be given at the main part of the article.  Theorem \ref{thm:pisier} is an optimal vector-valued version of a deep logarithmic Sobolev inequality of Talagrand \cite{Tal93}.  In the case of the uniform measure which corresponds to $\alpha=\tfrac{1}{2}$, Theorem \ref{thm:pisier} was proven recently in \cite{CE24}. 


\subsection{Metric stable type}

Besides its interest in the context of embedding theory, the work \cite{IVV20} of Ivanisvili, van Handel and Volberg was a major contribution to a long-standing research program in nonlinear functional analysis called the \emph{Ribe program}.  Put simply, the foundational idea of the program (as put forth by Bourgain \cite{Bou86}, who was inspired by a landmark rigidity theorem of Ribe \cite{Rib76}) is that isomorphic local properties of normed spaces can be equivalently reformulated using only distances between pairs of points, with no reference to the spaces' linear structure (see also \cite{Nao12,Bal13}). In this sense,  the main result of \cite{IVV20} completes the Ribe program for Rademacher type by proving that it is equivalent to nonlinear Enflo type.

Among the various isomorphic properties studied in the local theory of normed spaces,  \emph{stable type} (see \cite{Pis74a,Pis74b}) plays a prominent role, in particular due to its relation to the problem of identifying $\ell_p^n$ subspaces of Banach spaces \cite{Pis83}.   Recall that a symmetric random variable $\theta$ is distributed according to the standard $p$-stable law, where $p\in(0,2]$, if it satisfies $\mb{E} \exp(it\theta) = \exp(-|t|^p)$ for every $t\in\R$.  A Banach space $(X,\|\cdot\|_X)$ is said to have stable type $p$ if for every $r<p$, every $n\in\N$ and every vectors $v_1,\ldots,v_n\in X$,  we have 
\begin{equation}
\Big( \mb{E} \Big\| \sum_{i=1}^n \theta_i v_i\Big\|_X^r \Big)^{1/r} \leq \msf{C}_{p,r}(X) \Big( \sum_{i=1}^n \|v_i\|_X^p\Big)^{1/p},
\end{equation}
where the constant $\msf{C}_{p,r}(X)$ depends only on $p, r$ and $X$, but not on $n$ or the choice of $v_1,\ldots,v_n$, and $\theta_1,\theta_2,\ldots$ are independent standard symmetric $p$-stable random variables. 

In the special case $p=2$, $p$-stable random variables are Gaussians and thus stable type 2 coincides with Rademacher type 2.  Moreover, any space with stable type $p\in[1,2)$ also has Rademacher type $p$ but the converse does not necessarily hold.  In fact,  early results of Maurey and Pisier in the theory of type (see \cite[Proposition~3]{Pis74a} and \cite[Th\'eor\`eme~1]{Pis74b}), show that a normed space $X$ has stable type $p\in[1,2)$ if and only if $X$ has Rademacher type $p+\e$ for some $\e>0$. Therefore, formally, the Ribe program for stable type $p\in[1,2)$ has been completed by \cite{IVV20}: a normed space $X$ has stable type $p$ if and only if $X$ has Enflo type $p+\e$ for some $\e>0$.  In other words,  stable type $p$ of a normed space $X$ can be characterized by the validity of some, a priori unknown, $X$-valued Poincar\'e-type inequality from a given family.

To amend this, we shall introduce a new nonlinear invariant called \emph{metric stable type} which is again a Poincar\'e-type inequality for functions defined on the discrete hypercube and will serve as a metric characterization of stable type for normed spaces.  If $(\MM,d_\MM)$ is a metric space and $f:\{-1,1\}^n\to \MM$ is an $\MM$-valued function, we denote
\begin{equation}
\forall \ x\in\{-1,1\}^n, \qquad  \mf{d}_if(x) \eqdef \tfrac{1}{2}d_\MM\big(f(x), f\big(x(1),\ldots,-x(i),\ldots,x(n)\big)\big).
\end{equation}
Moreover,  we shall denote the weak $\ell_p$ norm on $\R^n$ by
\begin{equation}
\forall \ w\in\R^n,\qquad \|w\|_{\ell_{p,\infty}^n} \eqdef \sup_{r\geq0} \big\{ r\cdot \# \{i: \ |w_i|\geq r\}^{1/p}\big\},
\end{equation}
where $\# S$ denotes the cardinality of a finite set $S$.

\begin{definition}
A metric space $(\MM,d_\MM)$ has metric stable type $p\in(0,2)$ with constant $S\in(0,\infty)$ if for any $n\in\N$, every function $f:\{-1,1\}^n\to\MM$ satisfies
\begin{equation} \label{eq:st}
\int_{\{-1,1\}^n} d_\MM\big( f(x),f(-x)\big)^p \,\diff\sigma_n(x) \leq S^p  \int_{\{-1,1\}^n} \big\| \big(\mf{d}_1f(x),\ldots,\mf{d}_nf(x)\big)\big\|^p_{\ell_{p,\infty}^n} \,\diff\sigma_n(x).
\end{equation}
\end{definition}

Since $\|\cdot\|_{\ell_{p,\infty}^n} \leq\|\cdot\|_{\ell_p^n}$ by Markov's inequality,  any space with metric stable type $p$ also has Enflo type $p$ with the same constant.  We shall prove the following metric characterization.

\begin{theorem} \label{thm:stable}
A normed space $X$ has metric stable type $p\in[1,2)$ if and only if $X$ has stable type $p$.
\end{theorem}

Pisier's K-convexity theorem \cite{Pis82}, asserts that a Banach space $X$ is K-convex if and only if it has Rademacher type $p$ for some $p>1$. In view of the aforementioned characterization of stable type \cite{Pis74a, Pis74b}, this is equivalent to $X$ having stable type 1 and thus we derive the following metric characterization of $K$-convexity by a Poincar\'e inequality.

\begin{corollary}
A normed space $X$ is K-convex if and only if $X$ has metric stable type $1$.
\end{corollary}


\subsubsection*{Embeddings}

The fact that stable type of Banach spaces is a refinement of Rademacher type can also be seen at the level of (linear) embeddings. Indeed, if $X$ is a space of Rademacher type $p$, then standard considerations show that the Banach--Mazur distance between $\ell_1^n$ and any $n$-dimensional subspace of $X$ is at least a constant multiple of $n^{1-1/p}$.  This estimate is sharp for $X=\ell_p$.  On the other hand, if $X$ has stable type $p$, then the characterization of \cite{Pis74a,Pis74b} implies that $X$ has Rademacher type $r$ for some $r>p$.  Hence, there exists $\e(X)>0$ such that the Banach--Mazur distance between $\ell_1^n$ and any $n$-dimensional subspace of $X$ is at least $n^{1-1/p+\e(X)}$.

In the metric setting,  Enflo type gives distortion lower bounds for embeddings of $\{-1,1\}^n$.  More generally,  given a vector ${\bf w}=(w_1,\ldots,w_n) \in\R_+^n$, consider the metric space $\{-1,1\}^n_{{\bf w}}$ which is the hypercube $\{-1,1\}^n$ equipped with the \emph{weighted} Hamming metric
\begin{equation}
\forall \ x,y\in\{-1,1\}^n,\qquad \rho_{{\bf w}} (x,y) = \frac{1}{2} \sum_{i=1}^n w_i |x(i)-y(i)|.
\end{equation}
It follows readily from the definitions that if $\MM$ has Enflo type $p$, then any embedding of $\{-1,1\}^n_{{\bf w}}$ into $\MM$ incurs distortion at least a constant multiple of $\|{\bf w}\|_{\ell_1^n}/\|{\bf w}\|_{\ell_{p}^n}$.  For spaces of metric stable type $p$, we have the following improvement for the distortion of weighted hypercubes.

\begin{proposition} \label{prop:dist}
If a metric space $(\MM,d_\MM)$ has metric stable type $p$ with constant $S$, then
\begin{equation} \label{eq:dist}
\forall \ {\bf w}\in\R_+^n,\qquad \msf{c}_\MM\big(\{-1,1\}^n_{{\bf w}} \big) \geq \frac{\|{\bf w}\|_{\ell_1^n}}{S\|{\bf w}\|_{\ell_{p,\infty}^n}}.
\end{equation}
\end{proposition}

Concretely, choosing the weight vector ${\bf w}$ with $w_i = i^{-1/p}$, we have that $\|{\bf w}\|_{\ell_{p,\infty}^n} \asymp 1$, whereas $\|{\bf w}\|_{\ell_{p}^n} \asymp (\log n)^{1/p}$, therefore any embedding of $\{-1,1\}^n_{{\bf w}}$ into a metric space of Enflo type $p$ incurs distortion at least $n^{1-1/p}/(\log n)^{1/p}$ instead of the asymptotically stronger lower bound $n^{1-1/p}$ which one gets for target spaces of metric stable type $p$.


\subsubsection*{On the nonlinear Maurey--Pisier problem for type}

A landmark theorem of Maurey and Pisier \cite{MP76} asserts that if $p_X\in[1,2]$ is the supremal Rademacher type of an infinite-dimensional Banach space $X$, then $X$ containts the spaces $\ell_{p_X}^n$ uniformly.   In \cite{Pis83}, Pisier used stable type to give a much simpler proof of this important theorem, while simultaneously obtaining the following quantitative refinement: for any $\e>0$ and $p\in[1,2)$,  every normed space $X$ containts a subspace which is $(1+\e)$-isomorphic to $\ell_p^d$ provided that $d$ is smaller than some explicit (unbounded) function of $\e$ and the stable type $p$ constant of $X$.  When $X$ is infinite-dimensional, this yields the result of \cite{MP76} since $X$ does not have stable type $p_X$ by the results of \cite{Pis74a,Pis74b}.

In \cite{BMW86}, Bourgain, Milman and Wolfson introduced a notion of nonlinear type,  which is now referred to as BMW type, and showed that if a metric space $\msf{M}$ does not have BMW type $p$ for any $p>1$, then $\msf{c}_\MM(\{-1,1\}^n) = 1$ for any $n\in\N$.  Finding a satisfactory nonlinear Maurey--Pisier theorem for metric spaces of supremal nonlinear type $p>1$ remains an open problem (see also \cite[Section~6]{Nao14}). It would be interesting to understand whether the newly introduced notion of metric stable type can lead to a nonlinear version of the result of \cite{Pis83}.  We refer to Section \ref{sec:disc} for further remarks and open problems which naturally arise from this work.


\subsection*{Acknowledgements} I wish to thank Florent Baudier, Paata Ivanisvili and Assaf Naor for their constructive feedback on this work.


\section{Preliminaries} \label{sec:prel}

\subsection{Probability} In this section, we outline the basics of analysis on the biased hypercube, with an emphasis on the underlying semigroup structure.


\subsubsection*{The biased measure}

Recall that, for $\alpha\in(0,1)$,  the $\alpha$-biased probability measure $\mu_\alpha$ on $\{-1,1\}$ is given by $\mu_\alpha\{1\}=\alpha$ and $\mu_\alpha\{-1\}=1-\alpha$.  Moreover, if $\pmb{\alpha} = (\alpha_1,\ldots,\alpha_n)\in(0,1)^n$, then we shall denote by $\mu_{\pmb{\alpha}}$ the product measure $\mu_{\alpha_1}\otimes\cdots\otimes\mu_{\alpha_n}$ on the hypercube $\{-1,1\}^n$.


\subsubsection*{The Markov process}

For $\alpha\in(0,1)$, consider the transition matrices $\{p_t^\alpha\}_{t\geq0}$ on $\{-1,1\}$ given by
\begin{equation} \label{eq:transition}
\forall \ t\geq0,\qquad
\begin{pmatrix}
p_t^\alpha(1,1) & p_t^\alpha(1,-1) \\
p_t^\alpha(-1,1) & p_t^\alpha(-1,-1) 
\end{pmatrix}	
 = 
 \begin{pmatrix}
1-(1-e^{-t})(1-\alpha) & (1-e^{-t})(1-\alpha) \\
(1-e^{-t})\alpha & 1-(1-e^{-t})\alpha 
\end{pmatrix}	
\end{equation}
Moreover, for $\pmb{\alpha} = (\alpha_1,\ldots,\alpha_n)\in(0,1)^n$ consider the corresponding tensor products $\{p_t^{\pmb{\alpha}}\}_{t\geq0}$ on $\{-1,1\}^n$ given by
\begin{equation} \label{eq:transition2}
\forall \ x,y\in\{-1,1\}^n, \qquad p_t^{\pmb{\alpha}}(x,y) = \prod_{i=1}^n p_t^{\alpha_i}\big(x(i),y(i)\big).
\end{equation}
As each $p_t^{\alpha_i}$ is a row-stochastic $2\times2$ matrix with nonnegative entries, the same holds also for the $2^n\times2^n$ matrices $p_t^{\pmb{\alpha}}$. Therefore, $\{p_t^{\pmb{\alpha}}\}_{t\geq0}$ is the transition kernel of a time-homogeneous Markov chain $\{X_t^{\pmb{\alpha}}\}_{t\geq0}$ on $\{-1,1\}^n$, that is
\begin{equation}
\forall \ t, s\geq0,\qquad \mb{P}\big\{X_{t+s}^{\pmb{\alpha}} = y \ \big| \ X_s^{\pmb{\alpha}}=x \big\} = p_t^{\pmb{\alpha}}(x,y),
\end{equation}
where $x,y\in\{-1,1\}^n$. We shall need the following simple facts for this process.

\begin{lemma}
Fix $\pmb{\alpha}\in(0,1)^n$ and let $\{X_t^{\pmb{\alpha}}\}_{t\geq0}$ be a Markov process on $\{-1,1\}^n$ with transition kernels $\{p_t^{\pmb{\alpha}}\}_{t\geq0}$. Then, $\{X_t^{\pmb{\alpha}}\}_{t\geq0}$ is stationary and reversible with respect to $\mu_{\pmb{\alpha}}$.
 \end{lemma}
 
 \begin{proof}
Due to the product structure of the Markov chain, it suffices to consider the case $n=1$, that is, to prove that for $\alpha\in(0,1)$,
\begin{equation}
\forall \ x,y\in\{-1,1\},\qquad \mu_\alpha(x) p_t^\alpha(x,y) = \mu_\alpha(y) p_t^\alpha(y,x).
\end{equation} 
This follows automatically by the expression \eqref{eq:transition} for the transition matrix. The simple fact that reversibility implies stationarity is well-known \cite[Proposition~1.20]{LLP17}.
 \end{proof}
 
 The stationary Markov process $\{X_t^{\pmb{\alpha}}\}_{t\geq0}$ has a simple probabilistic interpretation which we shall now describe. For $i=1,\ldots,n$, let $\{N_t(i)\}_{t\geq0}$ be $n$ independent Poisson processes of unit rate and suppose that $X_0^{\pmb{\alpha}}$ is sampled from $\mu_{\pmb{\alpha}}$ independently of $\{N_t\}_{t\geq0}$. Then, at any time $t>0$ for which the process $N_t(i)$  jumps for some $i\in\{1,\ldots,n\}$, the corresponding value $X_t^{\pmb{\alpha}}(i)$ is updated independently from $\mu_{\alpha_i}$. An explicit calculation shows that this probabilistic construction gives rise exactly to the transition kernel of \eqref{eq:transition} and \eqref{eq:transition2}.


\subsubsection*{The corresponding semigroup}

Fix $\pmb{\alpha}\in(0,1)^n$ and let $\{P_t^{\pmb{\alpha}}\}_{t\geq0}$ be the Markov semigroup associated to the process $\{X_t^{\pmb{\alpha}}\}_{t\geq0}$. Concretely, if $X$ is a vector space, then for every function $f:\{-1,1\}^n\to X$ and $t\geq0$, we denote by
\begin{equation} \label{eq:sgp}
\forall \ x\in\{-1,1\}^n,\qquad P_t^{\pmb{\alpha}}f(x) = \mb{E}\big[ f\big(X_t^{\pmb{\alpha}}\big) \ \big| \ X_0^{\pmb{\alpha}} = x\big].
\end{equation}
In view of the above interpretation of $\{X_t^{\pmb{\alpha}}\}_{t\geq0}$ by means of a Poisson process, the action of the semigroup $\{P_t^{\pmb{\alpha}}\}_{t\geq0}$ can be computed via the identity
\begin{equation} \label{eq:sgp2}
\begin{split}
P_t^{\pmb{\alpha}}f & = \sum_{S\subseteq\{1,\ldots,n\}} \mb{P}\big\{N_t(i)>0 \ \mbox{for } i\in S \mbox{ and } N_t(i)=0 \mbox{ for } i\notin S\big\} \int_{\{-1,1\}^S} f \, \diff\prod_{i\in S} \mu_{\alpha_i}
\\ & = \sum_{S\subseteq\{1,\ldots,n\}} (1-e^{-t})^{|S|} e^{-t(n-|S|)}  \int_{\{-1,1\}^S} f \, \diff\prod_{i\in S} \mu_{\alpha_i}.
\end{split}
\end{equation}

\begin{lemma}
Fix $\pmb{\alpha}=(\alpha_1,\ldots,\alpha_n)\in(0,1)^n$ and let $\{P_t^{\pmb{\alpha}}\}_{t\geq0}$ be the semigroup \eqref{eq:sgp}. Then, the action of its generator $\ms{L}_{\pmb{\alpha}}$ on a function $f:\{-1,1\}^n\to X$, where $X$ is a vector space, is given by
\begin{equation} \label{eq:gen}
\forall \ x\in\{-1,1\}^n,\qquad \ms{L}_{\pmb{\alpha}} f(x) = -\sum_{i=1}^n \partial_i^{\alpha_i} f(x)
\end{equation}
where $\partial_i^\beta f(x) = f(x) - \int_{\{-1,1\}} f(x_1,\ldots,x_{i-1},y,x_{i+1},\ldots,x_n) \, \diff\mu_\beta(y)$ for $\beta\in(0,1)$.
\end{lemma}

\begin{proof}
The claim follows from the expression \eqref{eq:sgp2} of the semigroup and the definition
\begin{equation*}
\forall \ x\in\{-1,1\}^n,\qquad \ms{L}_{\pmb{\alpha}} f(x) = \frac{\diff}{\diff t}\Big|_{t=0} P_t^{\pmb{\alpha}}f(x). \qedhere
\end{equation*}
\end{proof}


\subsection{Topology}

Apart from the probabilistic elements from analysis on biased hypercubes, the proof of Theorem \ref{thm:main} also has a crucial topological component, following an idea of \cite{Ole96}.


\subsubsection*{The Borsuk--Ulam theorem}

While the Poincar\'e-type inequality of Enflo for $X$-valued functions on $\{-1,1\}^n$ cannot capture the dimension of the target space $X$, a key part of the argument towards Theorem \ref{thm:main} is to show that there exists a $\dim(X)$-dimensional subcube of $\{-1,1\}^n$ along with a bias vector $\pmb{\alpha}$ for which the $\pmb{\alpha}$-biased Poincar\'e inequalities (see Theorem \ref{thm:poincare} below) on this subcube and its antipodal yield much better distortion lower bounds. This will be proven using the Borsuk--Ulam theorem from algebraic topology, see \cite{Mat03}.

\begin{theorem} [Borsuk--Ulam]
For every continuous function $g:\mb{S}^d\to\R^d$, where $d\in\N$, there exists a point $w\in\mb{S}^d$ such that $g(w)=g(-w)$. 
\end{theorem}


\subsubsection*{Multilinear extension and low-dimensional faces of the cube}

Every function $f:\{-1,1\}^n\to X$ admits a unique multilinear extension on the solid cube $[-1,1]^n$, given by
\begin{equation} \label{eq:extension}
\forall \ y\in[-1,1]^n, \qquad F(y) \eqdef \sum_{S\subseteq \{-1,1\}^n} \Big( \frac{1}{2^n} \sum_{x\in\{-1,1\}^n} f(x) w_S(x) \Big) w_S(y),
\end{equation}
where $w_S(a) = \prod_{i\in S} a_i$, which is usually referred to as the Fourier--Walsh expansion of $f$.  Extending $f$ to the continuous cube allows for the use of topological methods.  In what follows, we will exploit the fact that the cube $[-1,1]^n$ is equipped with a canonical CW complex structure. Concretely,  for $d\in\{1,\ldots,n\}$,  consider the subsets
\begin{equation}
\ms{C}_d^n = \big\{ x \in[-1,1]^n: \ \mbox{there exists } \sigma\subseteq\{1,\ldots,n\} \mbox{ with } |\sigma|\geq n-d \mbox{ and } |x(i)|=1, \ \forall \ i\in\sigma\big\}
\end{equation}
consisting of all $\ell$-dimensional faces of $[-1,1]^n$ for $\ell\leq d$, so that $\ms{C}_n^n=[-1,1]^n$ and $\ms{C}_0^n=\{-1,1\}^n$.  We shall use the following elementary topological fact (see \cite[Lemma~1]{Ole96}).

\begin{lemma} \label{lem:ole}
If $d<n$, there exists a continuous map $h_d:\mb{S}^d\to \ms{C}_d^n$ with $h_d(-x)=-h_d(x)$,   $\forall \ x\in\mb{S}^d$.
\end{lemma}

Combining this and the Borsuk--Ulam theorem, we deduce the following useful lemma.

\begin{lemma} \label{lem:ole2}
If $n,d\in\N$ with $d<n$,  then for every continuous function $F:\ms{C}_d^n\to \R^d$ there exists a point $z\in\ms{C}_d^n$ such that F(z)=F(-z).
\end{lemma}

\begin{proof}
Consider the function $g\eqdef F\circ h_d : \mb{S}^d\to\R^d$, where $h_d$ is the function of Lemma \ref{lem:ole}.  By the Borsuk--Ulam theorem and the oddness of $h_d$, there exists a point $w\in\mb{S}^d$ such that
\begin{equation}
F(h_d(w)) = g(w) = g(-w) = F(h_d(-w)) = F(-h_d(w))
\end{equation}
and the conclusion follows by choosing $z=h_d(w)\in\ms{C}_d^n$.
\end{proof}


\section{Proof of Theorem \ref{thm:main}}

We are now ready to proceed to the main part of the proof.  The main analytic component is a biased version of the key formula of \cite{IVV20} for the time derivative of the heat flow on $\{-1,1\}^n$.  Given $t\geq0$, $\alpha\in(0,1)$ and an auxiliary parameter $\theta\in\R$, consider the matrix $\eta_t^\alpha(\cdot,\cdot;\theta)$ given by
\begin{equation} \label{eq:eta}
\forall \ t\geq0,\qquad
\begin{pmatrix}
\eta_t^\alpha(1,1;\theta) & \eta_t^\alpha(1,-1;\theta) \\
\eta_t^\alpha(-1,1;\theta) & \eta_t^\alpha(-1,-1;\theta) 
\end{pmatrix}	
 = 
 \begin{pmatrix}
\frac{e^{-t}-\theta}{p_t^\alpha(1,1)} & \frac{-\theta}{p_t^\alpha(-1,1)} \\
\frac{\theta-e^{-t}}{p_t^\alpha(1,-1)} & \frac{\theta}{p_t^\alpha(-1,-1)} 
\end{pmatrix}.
\end{equation}
For future reference, we record the following straightforward properties of $\eta_t^\alpha(\cdot,\cdot;\theta)$.

\begin{lemma}
Fix $t\geq0$ and $\alpha\in(0,1)$. Then,
\begin{equation} \label{eq:center}
\forall \ x\in\{-1,1\},  \  \theta\in \R,  \qquad p_t^\alpha(x,1) \eta_t^\alpha(1,x;\theta) + p_t^\alpha(x,-1) \eta_t^\alpha(-1,x;\theta) = 0
\end{equation}
and
\begin{equation} \label{eq:second-mom}
\begin{split}
\min_{\theta\in\R} \max_{x\in\{-1,1\}} \Big\{ p_t^\alpha(x,1)  \eta_t^\alpha(1,x;\theta)^2 & + p_t^\alpha(x,-1)  \eta_t^\alpha(-1,x;\theta)^2 \Big\} \\ & = \frac{e^{-t}}{(e^t-1)(\sqrt{\alpha p_t^\alpha(-1,-1)}+\sqrt{(1-\alpha)p_t^\alpha(1,1)})^2} \leq \frac{1}{e^t-1}.
\end{split}
\end{equation}
\end{lemma}

\begin{proof}
The centering condition \eqref{eq:center} can be checked easily using the explicit formulas \eqref{eq:transition} and \eqref{eq:eta} of the matrices. For \eqref{eq:second-mom}, we compute that for any $\theta\in\R$,
\begin{equation}
\begin{split}
 \max_{x\in\{-1,1\}} \Big\{ p_t^\alpha(x,1)  \eta_t^\alpha(1,x;\theta)^2 & + p_t^\alpha(x,-1)  \eta_t^\alpha(-1,x;\theta)^2 \Big\} 
 \\ & = \max\Big\{ \frac{(e^{-t}-\theta)^2}{p_t^\alpha(1,1) p_t^\alpha(1,-1)},  \frac{\theta^2}{p_t^\alpha(-1,1) p_t^\alpha(-1,-1)},\Big\}.
\end{split}
\end{equation}
As this is the maximum of two quadratic functions in $\theta$, its minimum is attained at the point $\theta^\ast$ where they intersect in the interval $(0,e^{-t})$, namely at
\begin{equation}
\theta^\ast = \frac{e^{-t}\sqrt{\alpha p_t^\alpha(-1,-1)}}{\sqrt{\alpha p_t^\alpha(-1,-1)}+\sqrt{(1-\alpha) p_t^\alpha(1,1)}}.
\end{equation}
The first equality in \eqref{eq:second-mom} is immediate, whereas for the inequality we compute
\begin{equation}
\begin{split}
 \frac{e^{-t}}{(e^t-1)(\sqrt{\alpha p_t^\alpha(-1,-1)}+\sqrt{(1-\alpha)p_t^\alpha(1,1)})^2} & \leq \frac{e^{-t}}{(e^t-1)(\alpha p_t^\alpha(-1,-1)+(1-\alpha)p_t^\alpha(1,1))}
\\ & = \frac{e^{-t}}{(e^t-1)\big(1-(1-e^{-t})(\alpha^2+(1-\alpha)^2)\big)} \leq \frac{1}{e^t-1},
\end{split}
\end{equation}
where both inequalities follow from the convexity of $x\mapsto x^2$.
\end{proof}

The key technical ingredient in the proof of Theorem \ref{thm:main} is the following identity.

\begin{proposition} \label{prop:ident}
Fix $n\in\N$, $\pmb{\alpha}=(\alpha_1,\ldots,\alpha_n)\in(0,1)^n$, $t\geq0$ and $\theta_1,\ldots,\theta_n\in\R$. Then, for every function $f:\{-1,1\}^n\to X$, where $X$ is a vector space, we have
\begin{equation} \label{eq:identi}
\forall \ x\in\{-1,1\}^n,\qquad \ms{L}_{\pmb{\alpha}} P_t^{\pmb{\alpha}} f(x) = - \mb{E}\Big[ \sum_{i=1}^n \eta_t^{\alpha_i}\big(x(i),X_t^{\pmb{\alpha}}(i);\theta_i\big) \partial_i^{\alpha_i} f(X_t^{\pmb{\alpha}}) \ \Big| \ X_0^{\pmb{\alpha}}=x\Big].
\end{equation}
\end{proposition}

\begin{proof}
In view of \eqref{eq:gen} and the product structure of the process $\{X_t^{\pmb{\alpha}}\}_{t\geq0}$, it suffices to check the claim for $n=1$, namely that for every $\beta\in(0,1)$, $\theta\in\R$ and $f:\{-1,1\}\to X$,
\begin{equation}
\forall \ x\in\{-1,1\},\qquad e^{-t} \partial^\beta f(x) = P_t^\beta \partial^\beta f(x) = \mb{E}\big[ \eta_t^\beta(x,X_t^\beta;\theta) \partial^\beta f(X_t^\beta) \ \big| \ X_0^\beta=x\big],
\end{equation}
where the first equality follows from the probabilistic representation \eqref{eq:sgp2}. Taking into account that $\beta \partial^\beta f(1) + (1-\beta) \partial^\beta f(-1) = 0$,  this amounts to the system of equations
\begin{equation}
\begin{cases}
e^{-t} = p_t^\beta(1,1) \eta_t^\beta(1,1;\theta) - \frac{\beta}{1-\beta} p_t^\beta(1,-1) \eta_t^\beta(1,-1;\theta) \\
e^{-t} = -\frac{1-\beta}{\beta} p_t^\beta(-1,1) \eta_t^\beta(-1,1;\theta) + p_t^\beta(-1,-1) \eta_t^\beta(-1,-1;\theta)
\end{cases}
\end{equation}
which can be easily verified by direct computation.
\end{proof}

\begin{theorem} \label{thm:poincare}
Fix $p\in[1,2]$ and let $(X,\|\cdot\|_X)$ be a normed space of Rademacher type $p$.  Then, for any $n\in\N$ and $\pmb{\alpha}=(\alpha_1,\ldots,
\alpha_n)\in(0,1)^n$, every function $f:\{-1,1\}^n\to X$ satisfies
\begin{equation}
\int_{\{-1,1\}^n}\Big\| f(x) - \int_{\{-1,1\}^n} f \,\diff\mu_{\pmb{\alpha}} \Big\|_X^p \,\diff\mu_{\pmb{\alpha}}(x) \leq \big(2\pi \T_p(X)\big)^p \sum_{i=1}^n \int_{\{-1,1\}^n} \big\| \partial_i^{\alpha_i} f(x)\big\|_X^p \,\diff\mu_{\pmb{\alpha}}(x).
\end{equation}
\end{theorem}

\begin{proof}
Writing
\begin{equation}
f(x) - \int_{\{-1,1\}^n} f \,\diff\mu_{\pmb{\alpha}} = P_0^{\pmb{\alpha}}f(x) - P_\infty^{\pmb{\alpha}}f(x) = - \int_0^\infty \ms{L}_{\pmb{\alpha}} P_t^{\pmb{\alpha}}f(x)\,\diff t
\end{equation}
and using Jensen's inequality and Proposition \ref{prop:ident}, we see that for $\theta_1(t),\ldots,\theta_n(t)\in\R$,
\begin{equation} \label{eq:apply-sgp}
\begin{split}
\Bigg(\int_{\{-1,1\}^n}\Big\| &f(x)  - \int_{\{-1,1\}^n} f \,\diff\mu_{\pmb{\alpha}} \Big\|_X^p \,\diff\mu_{\pmb{\alpha}}(x) \Bigg)^{1/p}\leq \int_0^\infty \Bigg(\int_{\{-1,1\}^n}\Big\|  \ms{L}_{\pmb{\alpha}} P_t^{\pmb{\alpha}}f(x) \Big\|_X^p \,\diff\mu_{\pmb{\alpha}}(x)\Bigg)^{1/p} \,\diff t
\\ & = \int_0^\infty \Bigg(\int_{\{-1,1\}^n}\Bigg\| \mb{E}\Big[ \sum_{i=1}^n \eta_t^{\alpha_i}\big(x(i),X_t^{\pmb{\alpha}}(i);\theta_i(t)\big) \partial_i^{\alpha_i} f(X_t^{\pmb{\alpha}}) \ \Big| \ X_0^{\pmb{\alpha}}=x\Big] \Bigg\|_X^p \,\diff\mu_{\pmb{\alpha}}(x)\Bigg)^{1/p} \,\diff t
\\ & \leq \int_0^\infty \Bigg( \mb{E} \Big\| \sum_{i=1}^n \eta_t^{\alpha_i}\big(X_0^{\pmb{\alpha}}(i),X_t^{\pmb{\alpha}}(i);\theta_i(t)\big) \partial_i^{\alpha_i} f(X_t^{\pmb{\alpha}}) \Big\|_X^p \Bigg)^{1/p} \, \diff t,
\end{split}
\end{equation}
where in the last expectation $X_0^{\pmb{\alpha}}$ is distributed according to $\mu_{\pmb{\alpha}}$. Now, by the reversibility of the chain, this expectation can be written as
\begin{equation} \label{eq:use-rev}
\begin{split}
\mb{E} \Big\| \sum_{i=1}^n \eta_t^{\alpha_i}\big(X_0^{\pmb{\alpha}}(i),X_t^{\pmb{\alpha}}(i);&\theta_i(t)\big)  \partial_i^{\alpha_i} f(X_t^{\pmb{\alpha}}) \Big\|_X^p 
\\ & = \int_{\{-1,1\}^n} \sum_{y\in \{-1,1\}^n} p_t^{\pmb{\alpha}}(x,y) \Big\| \sum_{i=1}^n \eta_t^{\alpha_i}\big(y(i),x(i);\theta_i(t)\big) \partial_i^{\alpha_i} f(x)\Big\|_X^p \, \diff\mu_{\pmb{\alpha}}(x).
\end{split}
\end{equation}
Fixing $x\in\{-1,1\}^n$,  equation \eqref{eq:center} asserts that each $\eta_t^{\alpha_i} (y(i),x(i);\theta_i(t))$ is a centered random variable when $y(i)$ is distributed according to $p_t^{\alpha_i}(x(i),\cdot)$. Therefore,  as $p_t^{\pmb{\alpha}}(x,\cdot)$ is a product measure,  the Rademacher type condition for sums of centered independent random vectors (see \cite[Proposition~9.11]{LT91}) yields the bound
\begin{equation} \label{eq:use-type}
\begin{split}
& \int_{\{-1,1\}^n} \sum_{y\in \{-1,1\}^n} p_t^{\pmb{\alpha}}(x,y) \Big\| \sum_{i=1}^n \eta_t^{\alpha_i}\big(y(i),x(i);\theta_i(t)\big) \partial_i^{\alpha_i} f(x)\Big\|_X^p \, \diff\mu_{\pmb{\alpha}}(x)
\\ & \leq \big( 2\T_p(X)\big)^p \int_{\{-1,1\}^n} \sum_{i=1}^n \sum_{y(i)\in \{-1,1\}} p_t^{\alpha_i}\big(x(i),y(i)\big)\ \big|\eta_t^{\alpha_i}\big(y(i),x(i);\theta_i(t)\big)\big|^p \big\|\partial_i^{\alpha_i} f(x)\big\|_X^p \, \diff\mu_{\pmb{\alpha}}(x)
\\ & \leq \big( 2\T_p(X)\big)^p  \sum_{i=1}^n \int_{\{-1,1\}^n} \Big(\sum_{y(i)\in \{-1,1\}} p_t^{\alpha_i}\big(x(i),y(i)\big)\ \big|\eta_t^{\alpha_i}\big(y(i),x(i);\theta_i(t)\big)\big|^2\Big)^{p/2} \big\|\partial_i^{\alpha_i} f(x)\big\|_X^p \, \diff\mu_{\pmb{\alpha}}(x),
\end{split}
\end{equation}
where we also used that $p\leq2$.  Now, choosing the $\theta_i(t)$ which minimize the quantity in the left-hand side of \eqref{eq:second-mom} with bias $\alpha_i$,  and combining \eqref{eq:apply-sgp}, \eqref{eq:use-rev} and \eqref{eq:use-type}, we conclude that
\begin{equation}
\begin{split}
 \Bigg(\int_{\{-1,1\}^n}\Big\| f(x)  - & \int_{\{-1,1\}^n} f \,\diff\mu_{\pmb{\alpha}} \Big\|_X^p \,\diff\mu_{\pmb{\alpha}}(x) \Bigg)^{1/p}
\\ & \leq 2\T_p(X) \int_0^\infty  \Bigg( \sum_{i=1}^n \int_{\{-1,1\}^n} \big\| \partial_i^{\alpha_i} f(x)\big\|_X^p \,\diff\mu_{\pmb{\alpha}}(x)   \Bigg)^{1/p} \, \frac{\diff t}{\sqrt{e^t-1}},
\end{split}
\end{equation}
which is precisely the desired estimate.
\end{proof}

Equipped with the biased Poincar\'e inequality of Theorem \ref{thm:poincare}, we can conclude the proof.

\begin{proof} [Proof of Theorem \ref{thm:main}]
Let $X=(\R^d,\|\cdot\|_X)$ be a $d$-dimensional normed space and suppose that $f:\{-1,1\}^n\to X$ is a function such that
\begin{equation} \label{eq:lip-con}
\forall \ x,y\in\{-1,1\}^n, \qquad \rho(x,y) \leq \|f(x)-f(y)\|_X \leq D\rho(x,y)
\end{equation}
for some $D\geq1$.  The conclusion of the theorem follows from \cite{IVV20} when $d\geq n$ so we shall assume that $d<n$.  Let $F:[-1,1]^n\to X$ be the multilinear extension of $f$ given by \eqref{eq:extension}. Then, $F$ is clearly continuous as a polynomial and therefore, by Lemma \ref{lem:ole2},  there exists a point $z\in\ms{C}_d^n$ such that $F(z)=F(-z)$.  As $z$ has at least $n-d$ coordinates equal to 1 in absolute value we shall assume without loss of generality that $|z({d+1})|=\ldots=|z(n)|=1$ and consider the functions $h_+, h_-:\{-1,1\}^d\to X$ which are defined as
\begin{equation}
\forall \ x\in\{-1,1\}^d, \qquad h_\pm(x) = f\big( \pm x(1),\ldots, \pm x(d),\pm z(d+1),\ldots,\pm z(n)\big).
\end{equation}
Consider also the bias vector $\pmb{\alpha}_z = \big( \frac{1+z(1)}{2},\ldots,\frac{1+z(d)}{2}\big)\in(0,1)^d$ and notice that, by the multilinearity of $F$, we have the identity
\begin{equation}
\int_{\{-1,1\}^d} h_+(x)\,\diff\mu_{\pmb{\alpha}_z}(x) = F(z) = F(-z) = \int_{\{-1,1\}^d} h_-(x)\,\diff\mu_{\pmb{\alpha}_z}(x).
\end{equation}
Therefore,  by the triangle inequality and Theorem \ref{thm:poincare} we get
\begin{equation} \label{eq:main-ineq}
\begin{split}
\int_{\{-1,1\}^d}\big\|  h_+(x) & - h_-(x) \big\|_X^p \,\diff\mu_{\pmb{\alpha}_z}(x)
\\ & \leq 2^{p-1} \int_{\{-1,1\}^d}\big\| h_+(x) - F(z) \big\|_X^p + \big\| h_-(x) - F(-z) \big\|_X^p \,\diff\mu_{\pmb{\alpha}_z}(x)
\\ & \leq 2^{2p-1} \big(\pi \T_p(X)\big)^p \sum_{i=1}^d \int_{\{-1,1\}^d} \big\| \partial_i^{\frac{1+z(i)}{2}} h_+(x)\big\|_X^p+\big\| \partial_i^{\frac{1+z(i)}{2}} h_-(x)\big\|_X^p \,\diff\mu_{\pmb{\alpha}_z}(x)  .
\end{split}
\end{equation}
Now, in view of the lower Lipschitz condition \eqref{eq:lip-con}, we clearly have
\begin{equation*}
\big\|  h_+(x) - h_-(x) \big\|_X = \big\|  f\big(x(1),\ldots,x(d),z(d+1),\ldots,z(n)\big) - f\big(-x(1),\ldots,-x(d),-z(d+1),\ldots,-z(n)\big) \big\|_X \geq n
\end{equation*}
for every $x\in\{-1,1\}^d$. On the other hand, for a fixed $i\in\{1,\ldots,d\}$ and $\beta = \frac{1+z(i)}{2}$, we have
\begin{equation*}
\begin{split}
\int_{\{-1,1\}} \big\| \partial_i^\beta & h_+(x)\big\|_X^p\,\diff\mu_\beta(x(i)) = \beta \|\partial_i^\beta h_+(x(1),\ldots,1,\ldots,x(d))\|_X^p + (1-\beta) \|\partial_i^\beta h_+(x(1),\ldots,-1,\ldots,x(d))\|_X^p
\\ & = \big( \beta(1-\beta)^p + (1-\beta) \beta^p \big) \ \big\|h_+(x(1),\ldots, 1,\ldots,x(d)) - h_+(x(1),\ldots, -1,\ldots,x(d))\big\|_X^p
\leq \frac{D^p}{2^p},
\end{split}
\end{equation*}
where in the last equality we used that $p\leq 2$ along with the upper Lipschitz condition \eqref{eq:lip-con}.  The same bound also holds for $h_-$.  Integrating the last two inequalities and combining them with \eqref{eq:main-ineq}, we deduce that
\begin{equation}
n^p \leq (2\pi \T_p(X))^p d D^p,
\end{equation}
which completes the proof of the theorem.
\end{proof}

\begin{remark}
The identity of Proposition \ref{prop:ident} in the case of the uniform measure $\sigma_n$ (which was obtained in \cite{IVV20}) is simpler.  Let $\xi_1(t),\ldots,\xi_n(t)$ be i.i.d.~random variables distributed according to $\mu_{\beta(t)}$, where $\beta(t) = \frac{1+e^{-t}}{2}$. Then, for any point $x\in\{-1,1\}^n$,  the corresponding unbiased process $\{X_t(i)\}_{t\geq0}$ with $X_0=x$ has distribution equal to $x(i)\xi_i(t)$ at time $t$.  Thus applying formula \eqref{eq:identi} with $\alpha_i = \frac{1}{2}$ and $\theta_i (t)= \tfrac{e^{-t}}{2}$, we recover the usual identity
\begin{equation} \label{eq:ivv}
\forall \ x\in\{-1,1\}^n,\qquad \ms{L}P_t f(x) = - \mb{E}\Big[  \sum_{i=1}^n \frac{\xi_i(t)-e^{-t}}{e^t-e^{-t}}\cdot  \partial_i f\big(x\xi(t)\big) \Big],
\end{equation}
where $x\xi(t) = (x(1)\xi_1(t),\ldots,x(n)\xi_n(t))$, as was proven in \cite{IVV20}.
\end{remark}


\section{Proof of Theorem \ref{thm:pisier}} \label{sec:pisier}

Recall that a normed space $(X,\|\cdot\|_X)$ has cotype $q\in[2,\infty)$ with constant $C\in(0,\infty)$ if for every $n\in\N$ and $v_1,\ldots,v_n\in X$, we have
\begin{equation}
\int_{\{-1,1\}^n} \Big\| \sum_{i=1}^n x_iv_i\Big\|_X^q \,\diff\sigma_n(x) \geq \frac{1}{C^q} \sum_{i=1}^n \|v_i\|_X^q.
\end{equation}
We say that $X$ has finite cotype if it has cotype $q$ for some $q\in[2,\infty)$. 

The proof of Theorem \ref{thm:pisier} is similar to the arguments of \cite{IVV20,CE24} using as input the new identity \eqref{eq:identi} for the time derivative of the biased semigroup $\{P_t^{\pmb{\alpha}}\}_{t\geq0}$.  In view of that, we shall omit various simple details and we will be less attentive with the values of the implicit constants. 

We start by proving the (weaker) biased Pisier inequality, in which the Orlicz norm on the left hand side of the conclusions of Theorem \ref{thm:pisier} is replaced by an $L_p$ norm.

\begin{theorem} \label{thm:pisier-weak}
For every $\alpha\in(0,1)$, there exists $\msf{K}_{\alpha}\in(0,\infty)$ such that the following holds when $p\in[1,\infty)$.  For any normed space $(X,\|\cdot\|_X)$ and any $n\in\N$,  every function $f:\{-1,1\}^n\to X$ satisfies
\begin{equation} \label{eq:pw1}
\bigg\| f - \int_{\{-1,1\}^n} f\,\diff\mu_\alpha^n\bigg\|_{L_p(\mu_\alpha^n;X)} \leq \msf{K}_{\alpha}(\log n+1) \left( \int_{\{-1,1\}^n} \Big\| \sum_{i=1}^n \delta_i \partial_i^\alpha f\Big\|_{L_p(\mu_\alpha^n;X)}^p \,\diff\sigma_n(\delta)\right)^{1/p}.
\end{equation}
If additionally $X$ is assumed to be of finite cotype, then there exists $\msf{K}_{p,\alpha}(X)$ such that
\begin{equation} \label{eq:pw2}
\bigg\| f - \int_{\{-1,1\}^n} f\,\diff\mu_\alpha^n\bigg\|_{L_p(\mu_\alpha^n;X)} \leq \msf{K}_{p,\alpha}(X)  \left( \int_{\{-1,1\}^n} \Big\| \sum_{i=1}^n \delta_i \partial_i^\alpha f\Big\|_{L_p(\mu_\alpha^n;X)}^p \,\diff\sigma_n(\delta)\right)^{1/p}.
\end{equation}
\end{theorem}

\begin{proof}
Let $\pmb{\alpha}=(\alpha,\ldots,\alpha)$ and without loss of generality assume that $\int f \,\diff\mu_{\pmb{\alpha}}=0$.  We start by proving the general inequality \eqref{eq:pw1}.  Notice that, by convexity, 
\begin{equation}
\big\| \ms{L}_{\pmb{\alpha}}f\big\|_{L_p(\mu_{\pmb{\alpha}};X)} \stackrel{\eqref{eq:gen}}{\leq} \sum_{i=1}^n \big\| \partial_i^\alpha f\big\|_{L_p(\mu_{\pmb{\alpha}};X)} \leq 2n \| f \|_{L_p(\mu_{\pmb{\alpha}};X)},
\end{equation}
where the last inequality follows from the definitions of $\partial_i^\alpha$. Therefore, 
\begin{equation*}
\begin{split}
\forall \ t\geq0,  \qquad \big\| e^{-t\ms{L}_{\pmb{\alpha}}} f\big\|_{L_p(\mu_{\pmb{\alpha}};X)} \leq \sum_{m=0}^\infty \frac{t^m}{m!} \big\| \ms{L}_{\pmb{\alpha}}^mf\big\|_{L_p(\mu_{\pmb{\alpha}};X)} \leq\sum_{m=0}^\infty \frac{(2nt)^m}{m!} \|f\|_{L_p(\mu_{\pmb{\alpha}};X)} = e^{2nt}\|f\|_{L_p(\mu_{\pmb{\alpha}};X)}.
\end{split}
\end{equation*}
Since $P_t^{\pmb{\alpha}} = e^{t\ms{L}_{\pmb{\alpha}}}$ for $t\geq0$,  this inequality applied to $P_t^{\pmb{\alpha}}f$ with $t=\tfrac{1}{n}$ implies that
\begin{equation}
\|f\|_{L_p(\mu_{\pmb{\alpha}};X)} \leq e^2 \big\|P_{1/n}^{\pmb{\alpha}}f\big\|_{L_p(\mu_{\pmb{\alpha}};X)}.
\end{equation}
Similarly to \eqref{eq:apply-sgp} and \eqref{eq:use-rev}, we thus have
\begin{equation}
\begin{split}
\|f&\|_{L_p(\mu_{\pmb{\alpha}};X)}\leq e^2 \big\|P_{1/n}^{\pmb{\alpha}}f\big\|_{L_p(\mu_{\pmb{\alpha}};X)}.
\\ & \leq
e^2 \int_{1/n}^\infty \left( \int_{\{-1,1\}^n} \sum_{y\in \{-1,1\}^n} p_t^{\pmb{\alpha}}(x,y) \Big\| \sum_{i=1}^n \eta_t^{\alpha}\big(y(i),x(i);\theta_i(t)\big) \partial_i^{\alpha} f(x)\Big\|_X^p \, \diff\mu_{\pmb{\alpha}}(x)\right)^{1/p} \,\diff t
\end{split}
\end{equation}
Fix $t\geq0$ and $x\in\{-1,1\}^n$.  Since $p_t^{\pmb{\alpha}}(x,\cdot)$ is a product measure,  the centering condition \eqref{eq:center} gives
\begin{equation} \label{use-cent}
\begin{split}
\sum_{y\in \{-1,1\}^n} p_t^{\pmb{\alpha}}(x,y) & \Big\| \sum_{i=1}^n \eta_t^{\alpha}\big(y(i),x(i);\theta_i(t)\big) \partial_i^{\alpha} f(x)\Big\|_X^p
\\ & \leq 2^p \int_{\{-1,1\}^n} \sum_{y\in \{-1,1\}^n} p_t^{\pmb{\alpha}}(x,y) \Big\| \sum_{i=1}^n \delta_i \big| \eta_t^{\alpha}\big(y(i),x(i);\theta_i(t)\big)\big| \partial_i^{\alpha} f(x)\Big\|_X^p \, \diff\sigma_n(\delta).
\end{split}
\end{equation}
To prove \eqref{eq:pw1}, set $\theta_i(t) = \frac{e^{-t}}{2}$ and observe that by the contraction principle \cite[Theorem~4.4]{LT91}, we can further bound this quantity by
\begin{equation}
\begin{split}
2^p \sum_{y\in \{-1,1\}^n} p_t^{\pmb{\alpha}}(x,y) & \int_{\{-1,1\}^n} \Big\| \sum_{i=1}^n \delta_i \big| \eta_t^{\alpha}\big(y(i),x(i);\tfrac{e^{-t}}{2}\big)\big| \partial_i^{\alpha} f(x)\Big\|_X^p \, \diff\sigma_n(\delta)
\\& \leq 2^p \max_{\chi,\psi\in\{-1,1\}} \big| \eta_t^{\alpha}\big(\psi,\chi;\tfrac{e^{-t}}{2}\big)\big|^p \int_{\{-1,1\}^n} \Big\| \sum_{i=1}^n \delta_i \partial_i^{\alpha} f(x)\Big\|_X^p \, \diff\sigma_n(\delta).
\end{split}
\end{equation}
It is now elementary to check that we have 
\begin{equation}
\max_{\chi,\psi\in\{-1,1\}}\big|\eta_t^{\alpha}\big(\psi,\chi;\tfrac{e^{-t}}{2}\big)\big| \lesssim_\alpha \begin{cases} \frac{1}{t}, & \mbox{for }t\in(0,1) \\ e^{-t}, & \mbox{for } t\geq1 \end{cases}
\end{equation}
and thus combining all the above we get
\begin{equation}
\|f\|_{L_p(\mu_{\pmb{\alpha}};X)} \lesssim_{\alpha} \left( \int_{1/n}^1 \frac{\diff t}{t} + \int_1^\infty e^{-t}\,\diff t \right)\cdot\left( \int_{\{-1,1\}^n} \Big\| \sum_{i=1}^n \delta_i \partial_i f\Big\|_{L_p(\mu_{\pmb{\alpha}};X)}^p \,\diff\sigma_n(\delta)\right)^{1/p},
\end{equation}
which readily concludes the proof of \eqref{eq:pw1}.

For the proof of \eqref{eq:pw2},  we combine \eqref{eq:apply-sgp} and \eqref{eq:use-rev} with \eqref{use-cent} to get
\begin{equation*}
\|f\|_{L_p(\mu_{\pmb{\alpha}};X)}\leq
2 \int_{0}^\infty \left( \int_{\{-1,1\}^{2n}}\sum_{y\in \{-1,1\}^n} p_t^{\pmb{\alpha}}(x,y) \Big\| \sum_{i=1}^n \delta_i \big| \eta_t^{\alpha}\big(y(i),x(i);\tfrac{e^{-t}}{2}\big)\big| \partial_i^{\alpha} f(x)\Big\|_X^p \, \diff\sigma_n(\delta) \diff\mu_{\pmb{\alpha}}(x)\right)^{1/p}\!\!\!\! \diff t.
\end{equation*}
Fix $x\in\{-1,1\}^n$ and $t\geq0$.  Denoting by $I(x)= \{i:\  x(i)=1\}$,  the inner term is bounded above by
\begin{equation} \label{eq:break-to-iid}
\begin{split}
 \int_{\{-1,1\}^n} & \sum_{y\in \{-1,1\}^n}  p_t^{\pmb{\alpha}}(x,y) \Big\| \sum_{i=1}^n \delta_i \big| \eta_t^{\alpha}\big(y(i),x(i);\tfrac{e^{-t}}{2}\big)\big| \partial_i^{\alpha} f(x)\Big\|_X^p \, \diff\sigma_n(\delta)
\\ & \leq 2^{p-1} \int_{\{-1,1\}^{I(x)}} \sum_{y\in \{-1,1\}^{I(x)}} p_t^{\pmb{\alpha}}(x,y) \Big\| \sum_{i\in I(x)} \delta_i \big| \eta_t^{\alpha}\big(y(i),1;\tfrac{e^{-t}}{2}\big)\big| \partial_i^{\alpha} f(x)\Big\|_X^p \, \diff\sigma_{I(x)}(\delta)
\\ & + 2^{p-1}  \int_{\{-1,1\}^{I(x)^c}}  \sum_{y\in \{-1,1\}^{I(x)^c}} p_t^{\pmb{\alpha}}(x,y) \Big\| \sum_{i\notin I(x)} \delta_i \big| \eta_t^{\alpha}\big(y(i),-1;\tfrac{e^{-t}}{2}\big)\big| \partial_i^{\alpha} f(x)\Big\|_X^p\, \diff\sigma_{I(x)^c}(\delta).
\end{split}
\end{equation}
When $y\in\{-1,1\}^{I(x)}$ is distributed according to $p_t^{\pmb{\alpha}}(x,\cdot)$ and $\delta$ is a uniformly random sign, the random variables $\delta_i|\eta_t^{\alpha}(y(i),1;\tfrac{e^{-t}}{2})|$,  $i\in I(x)$, are independent and identically distributed.  Therefore, standard comparison principles going back to works of Maurey and Pisier (see, e.g., \cite[Proposition~9.14]{LT91}) show that if $X$ has cotype $q<\infty$ and $r> \max\{p, q\}$, then
\begin{equation} \label{eq:Ix}
\begin{split}
\int_{\{-1,1\}^{I(x)}} &\sum_{y\in \{-1,1\}^{I(x)}}  p_t^{\pmb{\alpha}}(x,y) \Big\| \sum_{i\in I(x)} \delta_i \big| \eta_t^{\alpha}\big(y(i),1;\tfrac{e^{-t}}{2}\big)\big| \partial_i^{\alpha} f(x)\Big\|_X^p \, \diff\sigma_{I(x)}(\delta)
\\ & \lesssim_{r,X}  \Big( \sum_{\psi\in\{-1,1\}}p_t^{{\alpha}} \big(1,\psi;\tfrac{e^{-t}}{2}\big) \big| \eta_t^{\alpha}\big(\psi,1;\tfrac{e^{-t}}{2}\big)\big|^r \Big)^{p/r} \int_{\{-1,1\}^{I(x)}}  \Big\| \sum_{i\in I(x)} \delta_i \partial_i^{\alpha} f(x)\Big\|_X^p\, \diff\sigma_{I(x)}(\delta).
\end{split}
\end{equation}
It is again elementary to show that
\begin{equation} \label{rmom}
\max_{\chi\in\{-1,1\}} \Big( \sum_{\psi\in\{-1,1\}}p_t^{{\alpha}} \big(\chi,\psi;\tfrac{e^{-t}}{2}\big) \big| \eta_t^{\alpha}\big(\psi,\chi;\tfrac{e^{-t}}{2}\big)\big|^r \Big)^{1/r} \lesssim_{r,\alpha} \begin{cases} \frac{1}{t^{1-\frac1r}}, & \mbox{for }t\in(0,1) \\ e^{-t}, & \mbox{for } t\geq1 \end{cases}.
\end{equation}
In view of the convergence of the integrals
\begin{equation}
 \int_0^1 \frac{\diff t}{t^{1-\frac1r}} + \int_1^\infty e^{-t}\,\diff t  \lesssim_r 1,
\end{equation}
combining \eqref{eq:Ix}, the corresponding estimate on $I(x)^c$ and \eqref{eq:break-to-iid}, we finally get
\begin{equation*}
\begin{split}
\|f\|_{L_p(\mu_{\pmb{\alpha}};X)} \lesssim_{r,\alpha,X}   \left(\int_{\{-1,1\}^n} \int_{\{-1,1\}^n} \Big\| \sum_{i\in I(x)} \delta_i\partial_i^{\alpha} f(x)\Big\|_X^p + \Big\| \sum_{i\notin I(x)} \delta_i  \partial_i^{\alpha} f(x)\Big\|_X^p \,\diff\sigma_n(\delta) \, \diff\mu_{\pmb{\alpha}(x)} \right)^{1/p}.
\end{split}
\end{equation*}
Now, if $F,G$ are two independent centered $X$-valued random vectors, then
\begin{equation}
\mb{E}\|F+G\|_X^p \geq \max\big\{ \mb{E}_F \|F+\mb{E}_G[G]\|_X^p, \mb{E}_G\|\mb{E}_F[F]+G\|_X^p\big\} = \max\big\{ \mb{E}\|F\|_X^p, \mb{E}\|G\|_X^p\big\}.
\end{equation}
Therefore, 
\begin{equation}
\begin{split}
  \Bigg(\int_{\{-1,1\}^n} \int_{\{-1,1\}^n} \Big\| \sum_{i\in I(x)} \delta_i\partial_i^{\alpha} f(x)\Big\|_X^p  + \Big\| \sum_{i\notin I(x)} & \delta_i  \partial_i^{\alpha} f(x)\Big\|_X^p \,\diff\sigma_n(\delta) \, \diff\mu_{\pmb{\alpha}(x)} \Bigg)^{1/p}
\\ & \lesssim  \left(\int_{\{-1,1\}^n} \Big\| \sum_{i=1}^n \delta_i\partial_i^{\alpha} f(x)\Big\|_{L_p(\mu_{\pmb{\alpha}};X)}^p  \,\diff\sigma_n(\delta) \right)^{1/p}
\end{split}
\end{equation}
and this concludes the proof of the theorem.
\end{proof}

We now proceed to discuss the vector-valued $L_p$ logarithmic Sobolev inequality of Theorem \ref{thm:pisier}.  Recall that given a function $f:(\Omega,\mu)\to X$,  we define the $L_p(\log L)^a$ Orlicz norm of $f$ as
\begin{equation}
\|f\|_{L_p(\log L)^a(\mu;X)} = \inf\left\{ \gamma>0: \ \int_\Omega \tfrac{\|f(\omega)\|_X^p}{\gamma^p} \ \log^a\Big(e+\tfrac{\|f(\omega)\|_X^p}{\gamma^p}\Big) \,\diff\mu(\omega) \leq 1\right\}.
\end{equation}
In the case of the unbiased cube,  Theorem \ref{thm:pisier} was proven recently in \cite{CE24} and the proof used as a black box the unbiased analogue of Theorem \ref{thm:pisier-weak} which is due to \cite{Pis86,IVV20}. As the modifications which are required to derive Theorem \ref{thm:pisier} from Theorem \ref{thm:pisier-weak} in the biased case are almost mechanical,  we shall only offer a high-level description of the proof.

\begin{proof}[Sketch of the proof of Theorem \ref{thm:pisier}]
Given a scalar-valued function $h:\{-1,1\}^n\to\R$, we denote by $\msf{M}h:\{-1,1\}^n\to\R_+$ the asymmetric gradient of $h$, given by
\begin{equation}
\forall \ x\in\{-1,1\}^n,\qquad \msf{M}h(x) = \bigg( \sum_{i=1}^n \partial_ih(x)_+^2\bigg)^{1/2},
\end{equation}
where $a_+=\max\{a,0\}$ for $a\in\R$.  A combination of the triangle inequality with a technical result of Talagrand \cite[Proposition~5.1]{Tal93} on the biased cube implies that if a vector-valued function $f:\{-1,1\}^n\to X$ satisfies $\int f \,\diff\mu_\alpha^n=0$, then
\begin{equation}
\|f\|_{L_p(\log L)^{p/2}(\mu_\alpha^n;X)} \lesssim_{p,\alpha} \big\| \msf{M} \|f\|_X \big\|_{L_p(\mu_\alpha^n;X)} + \|f\|_{L_1(\mu_\alpha^n;X)}.
\end{equation}
The second term can be controlled by the right hand side of \eqref{eq:p1} or \eqref{eq:p2} using the results of Theorem \ref{thm:pisier-weak}.
 For the first term, we use the key pointwise inequality of \cite{CE24}, asserting that
 \begin{equation}
 \forall \ x\in\{-1,1\}^n,\qquad \msf{M}\|f\|_X(x) \leq \sqrt{2}\Bigg( \int_{\{-1,1\}^n} \Big\| \sum_{i=1}^n \delta_i \partial_i f(x)\Big\|_X^p\, \diff\sigma_n(\delta) \Bigg)^{1/p}.
 \end{equation}
To further upper bound this by the square function where $\partial_i$ is replaced by $\partial_i^\alpha$, observe that
\begin{equation}
\partial_i^\alpha f(x) \in\big\{ 2\alpha \partial_i f(x),  2(1-\alpha) \partial_if(x)\big\}
\end{equation}
and thus the contraction principle and integration in $x$ yield the desired conclusions.
\end{proof}


\section{Proof of Theorem \ref{thm:stable}} \label{sec:5}

In this section we shall prove the equivalence of stable type and metric stable type for norms.  The key ingredient in the proof is the following characterization of stable type (see, e.g., \cite[Proposition~9.12 (iii)]{LT91}) which goes back at least to \cite{MP84}.

\begin{lemma} \label{lem:lt}
A Banach space $(X,\|\cdot\|_X)$ has stable type $p\in[1,2)$ if and only if there exists a constant $\msf{ST}_p(X)\in(0,\infty)$ such that for every $n\in\N$ and every vectors $v_1,\ldots,v_n\in X$, we have
\begin{equation}  \label{eq:lt}
\int_{\{-1,1\}^n} \Big\| \sum_{i=1}^n x_i v_i\Big\|_X^p \,\diff\sigma_n(x) \leq \msf{ST}_p(X)^p \big\|\big(\|v_1\|_X,\ldots,\|v_n\|_X\big)\big\|_{\ell_{p,\infty}^n}^p.
\end{equation}
\end{lemma}

\begin{proof}[Proof of Theorem \ref{thm:stable}]
It is clear from Lemma \ref{lem:lt} that any normed space with metric stable type $p$ also has stable type $p$.  For the converse implication,  we treat the case $p=1$ separately.  Assuming that $X$ has stable type 1,  it follows from \cite{Pis74b} that it also has nontrivial Rademacher type and thus (in view of \cite{Pis73}) finite cotype.  Therefore, by the $X$-valued Pisier inequality with a dimension-free constant of \cite{IVV20}, we get
\begin{equation}
\begin{split}
\int_{\{-1,1\}^n} \big\|f(x)-f(-x)\big\|_X\,\diff\sigma_n(x) & \leq 2\int_{\{-1,1\}^n} \Big\| f(x) - \int_{\{-1,1\}^n} f\,\diff\sigma_n\Big\|_X \,\diff\sigma_n(x)
\\ & \lesssim_{X}  \int_{\{-1,1\}^{n}}\int_{\{-1,1\}^{n}} \mb{E} \Big\|\sum_{i=1}^{n} \delta_i \partial_i f(x)\Big\|_X \, \diff\sigma_n(\delta) \diff\sigma_{n}(x).
\end{split}
\end{equation}
Applying Lemma \ref{lem:lt} conditionally on $x\in\{-1,1\}^n$, we can further bound this quantity as
\begin{equation}
\begin{split}
 \int_{\{-1,1\}^{n}}\int_{\{-1,1\}^{n}} \mb{E} \Big\|\sum_{i=1}^{n} & \delta_i \partial_i f(x)\Big\|_X \,  \diff\sigma_n(\delta) \diff\sigma_{n}(x)
\\ & \leq \msf{ST}_1(X)  \int_{\{-1,1\}^n} \big\|\big(\|\partial_1f(x)\|_X,\ldots,\|\partial_nf(x)\|_X\big)\big\|_{\ell_{1,\infty}^n} \, \diff\sigma_n(x)
\end{split}
\end{equation}
and this proves the converse implication since $\mf{d}_if(x) =  \|\partial_if(x)\|_X$.

While this proof extends to all values of $p$,  in the case $p>1$ we present a more cumbersome argument which avoids the $X$-valued Pisier inequality and thus gives better dependence on parameters of $X$.  For $t\geq0$,  let $\xi_1(t),\ldots,\xi_n(t)$ be i.i.d.~random variables distributed according to $\mu_{\beta(t)}$, where $\beta(t)=\tfrac{1+e^{-t}}{2}$,  and denote by $\eta_i(t) = \frac{\xi_i(t)-e^{-t}}{e^t-e^{-t}}$.  Then,  it follows from the semigroup argument leading to \eqref{eq:apply-sgp} along with identity \eqref{eq:ivv} of \cite{IVV20} and Jensen's inequality that
\begin{equation}
\begin{split}
\Bigg(\int_{\{-1,1\}^n} \big\|f(x)-f(-x)\big\|_X^p\,\diff\sigma_n(x)\Bigg)^{1/p} & \leq 2\Bigg(\int_{\{-1,1\}^n} \Big\| f(x) - \int_{\{-1,1\}^n} f\,\diff\sigma_n\Big\|_X^p \,\diff\sigma_n(x)\Bigg)^{1/p}
\\ & \leq 2 \int_0^\infty \Bigg( \int_{\{-1,1\}^n} \mb{E} \Big\|\sum_{i=1}^n \eta_i(t) \partial_i f(x)\Big\|_X^p \, \diff\sigma_n(x)\Bigg)^{1/p} \,\diff t
\end{split}
\end{equation}

Fixing $x\in\{-1,1\}^n$,  the independent random vectors $\eta_1(t)\partial_1f(x),\ldots,\eta_n(t)\partial_nf(x)$ are centered.  Thus, by standard symmetrization estimates, we can further bound the last term by
\begin{equation}
\begin{split}
2 \int_0^\infty \Bigg( \int_{\{-1,1\}^n} \mb{E} \Big\|\sum_{i=1}^n \eta_i(t) \partial_i f(x)\Big\|_X^p& \, \diff\sigma_n(x)\Bigg)^{1/p} \,\diff t
\\ & \leq 4 \int_0^\infty \Bigg( \int_{\{-1,1\}^n} \mb{E} \Big\|\sum_{i=1}^n \delta_i \eta_i(t) \partial_i f(x)\Big\|_X^p \, \diff\sigma_n(x)\Bigg)^{1/p} \,\diff t,
\end{split}
\end{equation}
where the  expectation on the right hand side is with respect to $(\delta_i,\eta_i(t))$, where the $\delta_i$ are uniformly random signs, independent of the $\eta_i(t)$. Therefore, conditioning first on the values of $\eta_1(t),\ldots,\eta_n(t)$ and using Lemma \ref{lem:lt} for the integrand, we get
\begin{equation}
 \mb{E} \Big\|\sum_{i=1}^n \delta_i\eta_i(t) \partial_i f(x)\Big\|_X^p \leq \msf{ST}_p(X)^p \mb{E} \bigg\|\sum_{i=1}^n \eta_i(t) \big\| \partial_if(x)\big\|_X e_i \bigg\|_{\ell_{p,\infty}^n}^p,
\end{equation}
where $\{e_1,\ldots,e_n\}$ is the standard basis of $\R^n$. It is well-known that, since $p>1$,  $\ell_{p,\infty}^n$ is isomorphic to a normed space up to constants depending only on $p$. Moreover,  it has finite cotype with a constant independent of $n$ (for instance because it can be realized as a real interpolation space between $\ell_1^n$ and $\ell_2^n$, see \cite[Section~5.3]{BL76}).  Therefore,  using again the comparison principle \cite[Proposition~9.14]{LT91}, we have
\begin{equation}
\begin{split}
\mb{E} \bigg\|\sum_{i=1}^n \eta_i(t) \big\| \partial_if(x)\big\|_X e_i \bigg\|_{\ell_{p,\infty}^n}^p&  \lesssim_{p} \|\eta_1(t)\|_{L_r}^p \mb{E} \bigg\|\sum_{i=1}^n \delta_i \big\| \partial_if(x)\big\|_X e_i \bigg\|_{\ell_{p,\infty}^n}^p 
\\ & =  \|\eta_1(t)\|_{L_r}^p \big\|\big(\|\partial_1f(x)\|_X,\ldots,\|\partial_nf(x)\|_X\big)\big\|_{\ell_{p,\infty}^n}^p
\end{split}
\end{equation}
for any $r>\max\{r',p\}$, where $r'$ is the cotype of $\ell_{p,\infty}^n$.  Combining all the above along with the fact that $t\mapsto \|\eta_1(t)\|_{L_r}$ is integrable for any $r<\infty$,  we conclude the proof.  It is worth emphasizing that this proof shows additionally that the metric stable type constant of $X$ is proportional to the parameter $\msf{ST}_p(X)$ of \eqref{eq:lt} up to constants depending only on $p$, provided that $p>1$.
\end{proof}

Finally, we present the simple proof of the distortion bound \eqref{eq:dist}.

\begin{proof}[Proof of Proposition \ref{prop:dist}]
Suppose that $(\MM,d_\MM)$ has metric stable type $p$ with constant $S$ and suppose that a function $f:\{-1,1\}^n\to\MM$ satisfies
\begin{equation} \label{xx}
\forall \ x,y\in\{-1,1\}^n,\qquad s \rho_{\bf w}(x,y) \leq d_\MM\big(f(x),f(y)\big) \leq sD \rho_{\bf w}(x,y)
\end{equation} 
for some constants $s>0$ and $D\geq1$. Combined with the stable type assumption, this gives
\begin{equation}
\begin{split}
s^p \|{\bf w}\|_{\ell_1^n}^p  \stackrel{\eqref{xx}}{\leq} \int_{\{-1,1\}^n}  d_\MM\big( f(x)&,f(-x)\big)^p \,\diff\sigma_n(x) \leq S^p  \int_{\{-1,1\}^n} \big\| \big(\mf{d}_1f(x),\ldots,\mf{d}_nf(x)\big)\big\|^p_{\ell_{p,\infty}^n} \,\diff\sigma_n(x)
\\ &\stackrel{\eqref{xx}}{\leq} (sSD)^p  \int_{\{-1,1\}^n} \| (w_1,\ldots,w_n)\|^p_{\ell_{p,\infty}^n} \,\diff\sigma_n(x) = (sSD)^p \|{\bf w}\|_{\ell_{p,\infty}^n}^p.
\end{split}
\end{equation}
Rearranging gives the desired lower bound \eqref{eq:dist} for the distortion $D$.
\end{proof}


\section{Discussion and open problems} \label{sec:disc}

\noindent {\bf 1.} The proof of Theorem \ref{thm:main} in fact implies a Poincar\'e-type inequality for \emph{restrictions} of  functions $f:\{-1,1\}^n\to X$ if $\dim(X)<n$, which in turn yields the refined distortion lower bounds. An inspection of the argument reveals that for every such $f$ there exists a subset $\sigma\subseteq\{1,\ldots,n\}$ with $|\sigma| \leq \dim(X)$, a point $w\in\{-1,1\}^{\sigma^c}$ and a bias vector $\pmb{\alpha} = (\alpha_i)_{i\in\sigma}\in(0,1)^\sigma$ such that
\begin{equation} \label{eq:real-conclusion}
\begin{split}
\int_{\{-1,1\}^\sigma} \big\| f(x,w)-& f(-x,-w)\big\|_X^p \,\diff\mu_{\pmb{\alpha}}(x)
\\ & \leq 2^{2p-1} \big(\pi \T_p(X)\big)^p \sum_{i\in\sigma} \int_{\{-1,1\}^\sigma} \big\| \partial_i^{\alpha_i} f(x,w)\big\|_X^p + \big\|\partial_i^{\alpha_i}f(-x,-w)\big\|_X^p \,\diff\mu_{\pmb{\alpha}}(x).
\end{split}
\end{equation}

\noindent {\bf 2.} Such refinements of Poincar\'e-type inequalities for topological reasons had not been exploited since Oleszkiewicz's original work \cite{Ole96}. The last decades have seen the development of many metric inequalities on graphs which yield nonembeddability results into normed spaces. We believe that investigating whether the distortion estimates which one obtains this way can be further improved assuming upper bounds for the dimension of the target space is a very worthwhile research program. As examples, we mention the nonembeddability of graphs with large girth into uniformly smooth spaces \cite{LMN02,NPSS06}, of $\ell_\infty$-grids into spaces of finite cotype \cite{MN08} and of trees \cite{Bou86,LNP09} and diamond graphs \cite{JS09,EMN23} into uniformly convex spaces.

\medskip

\noindent {\bf 3.} The results of \cite{IVV20} in fact imply that any Lipschitz embedding of $\{-1,1\}^n$ into a normed space of Rademacher type $p$ incurs $p$-{\it average} distortion at least a constant multiple of $\T_p(X)^{-1} n^{1-1/p}$. It would be interesting to understand whether the bound of Theorem \ref{thm:main} can be extended to average distortion embeddings beyond bi-Lipschitz ones.

\medskip

\noindent {\bf 4.} The Poincar\'e inequality \eqref{eq:real-conclusion} implies that any $f:\{-1,1\}^n\to X=(\R^d,\|\cdot\|_X)$ satisfies
\begin{equation} \label{eq:edge-dist}
\sup_{(x,y): \ \rho(x,y)=n} \frac{n}{\|f(x)-f(y)\|_X} \sup_{(x,y): \ \rho(x,y)=1} \|f(x)-f(y)\|_X \gtrsim \frac{n}{\T_p(X) d^{1/p}},
\end{equation}
provided that $d<n$.  This is stronger than the lower bound of Theorem \ref{thm:main} as it implies that $f$ incurs large distortion only on \emph{specific} pairs of points, namely those which are either antipodal or connected by an edge.  Moreover though, \eqref{eq:edge-dist} is sharp for any value of $d$ and $p$.  Indeed, assume without loss of generality that $\tfrac{n}{d}$ is an odd integer and consider a partition  $I_1,\ldots,I_d$ of $\{1,\ldots,n\}$ in $d$ parts of size $\tfrac{n}{d}$.  Then, the mapping $f:\{-1,1\}^n\to\ell_p^d$, where $p\in[1,2]$, given by
\begin{equation}
\forall \ x\in\{-1,1\}^n, \qquad f(x) \eqdef \Big( \sum_{i\in I_1} x_i, \ldots, \sum_{i\in I_d} x_i\Big)
\end{equation}
satisfies
\begin{equation}
\inf_{(x,y): \ \rho(x,y)=n}\big\|f(x)-f(y)\big\|_{\ell_p^d} = \inf_{x\in\{-1,1\}^n} \bigg( \sum_{k=1}^d \Big| \sum_{i\in I_k} x_i\Big|^p\bigg)^{1/p} = d^{1/p}
\end{equation}
and $\|f(x)-f(y)\|_{\ell_p^d}=1$ when $\rho(x,y)=1$.

\medskip

\noindent {\bf 5.} The functional inequality \eqref{eq:real-conclusion} in fact implies that if $\theta\in(0,1)$,  then the bi-Lipschitz distortion of the $\theta$-snowflake of $\{-1,1\}^n$ into a finite dimensional normed space $X$ satisfies
\begin{equation}
\msf{c}_X\big(\{-1,1\}^n, \rho^\theta\big) \gtrsim \frac{n^\theta}{\T_p(X) \min\{n,\dim(X)\}^{1/p}}.
\end{equation}

\medskip

\noindent {\bf 6.} As the reasons behind the impossibility of dimension reduction of Theorem \ref{thm:main} are partly topological, it is natural to ask in what other settings can one deduce similar conclusions. For instance, do $d$-dimensional manifolds of nonpositive or nonnegative curvature (which have Enflo type 2, see \cite{Oht09}) admit the same distortion bounds as normed spaces of type 2?

\medskip

\noindent {\bf 7.} In this paper we extended the semigroup machinery developed in \cite{IVV20} to the biased cube, driven by the geometric application obtained in Theorem \ref{thm:main}.  As an aside, this led to biased versions of the vector-valued Poincar\'e and logarithmic Sobolev inequalities of \cite{IVV20,CE24}. The same reasoning also yields biased extensions of the vector-valued versions of Talagrand's influence inequality \cite{Tal94} which were studied in \cite{CE23} for the uniform probability measure on the hypercube.  Moreover,  combining the biased semigroup machinery with the arguments of \cite[Section~2.5]{BIM23}, one can derive biased versions of the isoperimetric-type inequalities of Eldan and Gross \cite{EG22}.  As these extensions are mostly mechanical given the material of this paper and in lack of a specific application which may follow from them, we omit them.

\bibliographystyle{plain}
\bibliography{flow-applications}

\end{document}